\documentclass[11pt]{amsart}
\usepackage[all]{xy}
\usepackage{amsmath}
\usepackage{amsfonts}
\usepackage{amssymb}
\usepackage{amscd}
\usepackage{amsthm}
\usepackage{latexsym}
\usepackage{amsbsy}
\usepackage{color,enumerate}

\def\0D{\Delta^{(0)}}
\def\1D{\Delta^{(1)}}

\def\Co{\,\square\,}

\newtheorem{theorem}{Theorem}[section]
\newtheorem{remark}[theorem]{Remark}
\newtheorem{proposition}[theorem]{Proposition}
\newtheorem{lemma}[theorem]{Lemma}
\newtheorem{corollary}[theorem]{Corollary}

\newtheorem{example}[theorem]{Example}

\def\build#1_#2^#3{\mathrel{\mathop{\kern 0pt#1}\limits_{#2}^{#3}}}
\newcommand{\ns}[1]{~\hspace{-3pt}_{\left<#1\right>}}
\newcommand{\ps}[1]{~\hspace{-3pt}^{^{\left(#1\right)}}}

\numberwithin{equation}{section}

 \newcommand{\ie}{{\it i.e.\/}\ }

\def\ve{\varepsilon}

\def\ot{\otimes}
\def\part{\partial}

\def\text{\hbox}

\def\ot{\otimes}

\def\Hom{\mathop{\rm Hom}\nolimits}

\def\Id{\mathop{\rm Id}\nolimits}

\def\build#1_#2^#3{\mathrel{
\mathop{\kern 0pt#1}\limits_{#2}^{#3}}}

\numberwithin{equation}{section}
\newcommand{\comment}[1]{\relax}

\begin{document}
\title{On Representation Theory of Total (Co)Integrals }
\author { Mohammad Hassanzadeh }
\curraddr{University of Windsor, Department of Mathematics and Statistics, Lambton Tower, Ontario, Canada.
}
\email{mhassan@uwindsor.ca}
\subjclass[2010]{06B15,  16T05, 11M55}
 \keywords{Representation theory, Hopf algebras, Noncommutative geometry}
\maketitle
\begin{abstract}

In this paper,  we show that total integrals and cointegrals  are new sources of stable anti Yetter-Drinfeld modules.
 We explicitly show that how special types of total  (co)integrals  can be used to provide both (stable) anti Yetter-Drinfeld   and Yetter-Drinfeld modules. We use these modules to classify total (co)integrals and (cleft) Hopf Galois (co)extensions for some examples of the Connes-Moscovici Hopf algebra, universal enveloping algebras and  polynomial algebras.

\end{abstract}

\section{ Introduction}
In this paper we study the representation theory of total integrals and total cointegrals. We show that  total (co)integrals are   new sources of stable anti Yetter-Drinfeld module.
This helps us to classify total (co)integrals and cleft Hopf Galois (co)extensions.  Stable anti Yetter-Drinfeld (SAYD) modules are suitable coefficients for Hopf cyclic homology \cite{CM98}, \cite{hkrs2} which is a strong algebraic tool in noncommutative geometry \cite{Connes book}. The way that total integrals can be used to produce Yetter-Drinfeld (YD) modules was noticed before in \cite{cfm}. In this paper we develop this idea. More precisely for any total (co)integral which is (co)algebra map we introduce an AYD and an YD module. Furthermore if the Hopf algebra $H$ has a modular pair in involution \cite{CM98} then we produce two  different AYD and YD modules. Then we replace the (co)algebra map property of the total (co)integral  by cleftness of a Hopf Galois (co)extension and similarly we construct all the four YD and AYD modules that we constructed in the previous case. The interesting fact here is that  the stability condition which can not  automatically be obtained for AYD modules constructed by total (co)integrals which are (co)algebra maps, can be obtained by cleft Hopf Galois (co)extensions for free. In fact if we have a  cleft Hopf Galois (co)extension then not only we obtain an AYD module defined by the total (co)integral but also we obtain the stability condition.

It is known that any   total (co)integral which is an (co)algebra map is convolution invertible. Conversely  total (co)integrals of any cleft Hopf Galois (co)extension have a close relation to some anti (co)algebra morphisms. We see that when  $H$ is (co)commutative any  Hopf Galois (co)extension provides a convolution invertible total integral. We use our AYD and YD modules  to classify total (co)integrals and Hopf Galois (co)extensions. As an example we see that there is no finite dimensional right (cleft) Hopf Galois extension over the Connes-Moscovici Hopf algebra. \\

 This paper is organized as follows:
 In the second section we review  the basics of total (co)integrals and cleft Hopf Galois (co)extensions.  In the third section we use total (co)integrals to produce (stable) anti Yetter-Drinfeld and Yetter-Drinfeld modules. For any total integral $f:H\longrightarrow A$ which is an algebra map we see that $A$ is an YD module  over $H$ in Proposition $3.2$ and  the quotient space $A_B= A/[A,B]$ is   an AYD module  over $H$ in Proposition $3.3$. If the Hopf algebra $H$ has a modular pair in involution then we construct a different AYD module in Proposition $3.4$ and a different YD module in Proposition $3.6$. Then we replace the algebra map property of the total integral by the condition that the Hopf Galois extension made by $(H,A)$
 is cleft.  In this case we recover all four YD and AYD modules that we constructed before and furthermore the AYD modules satisfy stability condition . Dually for a total cointegral $f: C\longrightarrow H$
 which is a coalgebra map we show that $C$ is an YD module over $H$ in Proposition $3.19$ and the subspace $C^D$ is an AYD module over $H$ in Proposition $3.20$.
 Again if the Hopf algebra $H$ has a modular pair in involution then we construct a different AYD module in Proposition $3.21$ and a different YD module in Remark $3.23$.
 Then we replace  the coalgebra map property of the total cointegral by the condition that Hopf Galois coextension made by $(H, C)$ is cleft.
 Similarly we recover all four YD and AYD modules that we constructed before.
 In the last section, we introduce several examples of our results for the Connes-Moscovici Hopf algebra,   universal enveloping algebras of a Lie algebra, and polynomial algebras.

\bigskip

\textbf{Acknowledgments}:
The author  deeply appreciates \emph{Atabey Kaygun} for his continues collaboration and ideas in the whole process of this work specially his carefully reading the manuscript. The author would also like to thank \emph{Gabriella Bohm} for  her valuable help  specially her idea of the statement and the proof of Lemma \ref{inverse-co}. At the end the author would like to appreciate  \emph{	Mihai Doru Staic} for his helpful comments on stable anti Yetter-Drinfeld modules.

\bigskip

\textbf{Notations}:
We denote a Hopf algebra by $H$ and we assume all Hopf algebras in this paper have an invertible antipode. All Hopf algebra, algebras and coalgebras in this paper are on a filed. We use the Sweedler notation $\Delta(c)=c\ps{1}\ot c\ps{2}$ for the coproduct of a coalgebra (Hopf algebra). The right coaction of a  Hopf algebra   on a comodule $M$ is denoted by summation notation $m\longmapsto m\ns{0}\ot m\ns{1}$. For all total integrals and cointegrals we assume that all modules or comodules over $H$ (co)act from right.
\tableofcontents

\section{Preliminaries}
In this section we review the properties of total (co)integrals specially the ones  which are (co)algebra morphisms and  cointegrals of cleft Hopf Galois (co)extensions.  For more information about this section see \cite{schn}, \cite{dmr}, \cite{DT}, \cite{bell}, \cite{cfm}.\\

 Let $H$ be Hopf algebra and $A$ be a right $H$-comodule algebra with the coaction $a\longmapsto a\ns{0}\ot a\ns{1}$. The coinvariant space of the coaction is the subalgebra $B=\{a\in A, ~ \rho(a)=a\ot1_H\}$. The extension $A(B)^H$ is called (right) Hopf Galois if the canonical map
$$ can\colon  A\ot_B A\longrightarrow A\ot H, \quad a\ot a'\longmapsto aa'\ns{0}\ot a'\ns{1},$$ is bijective.\\

Dually let $C$ be a right $H$-module coalgebra.  The set
$$I=span <  ch -\varepsilon(h)c>,$$ is a two-sided coideal of $C$ and $D=C/I$ is a coalgebra. The surjection  $\pi\colon C\to D$ defines a $C$-bicomodule structure on $D$. This coextension is called a (right) $H$-Galois  if the  canonical map
\begin{equation}
   can\colon  C\ot H \longrightarrow C\Box_D C, ~~~ c\ot h\longmapsto c\ps{1}\ot c\ps{2}h,
\end{equation}
is a bijection. We denote a Hopf Galois coextension by $C(D)_H$. \\

Here we recall the main object of the paper.  For any $H$-comodule algebra $A$ the map $f: H\longrightarrow A$ is called a total integral if it is a unital $H$-comodule map, \ie $f(1)=1$ and $f (h\ps{1})\ot h\ps{2}=f(h)\ns{0}\ot f(h)\ns{1}$. Dually for any $H$-module coalgebra $C$   a total cointegral   is  a counital $H$-module map $f \colon  C \longrightarrow H$  \ie     $\varepsilon(f(c))= \varepsilon(c)$ and $f(ch)= f(c)h$.\\

Let us recall that a Hopf Galois coextension $C(D)_H$ is called cleft  if there is a total cointegral $f\colon  C\longrightarrow H$    which is convolution invertible. This means that there exists a linear map $f^{-1}\colon  C\longrightarrow H$  such that $f\ast f^{-1}(c)=\varepsilon(c)1_H.$ By \cite[Lemma 2.3]{dmr}  counitality condition can be ignored for invertible cointegrals because if $f\colon  C\longrightarrow H$ is an invertible $H$-module map, then the right $H$-module map $f'\colon  C\longrightarrow H$ given by $f'(c)= \varepsilon(f^{-1}(c\ps{1})f(c\ps{2})$ is a counital invertible total cointegral. Dually a Hopf Galois extension  $A(B)^H$ is called cleft if there is a total integral $f: H\longrightarrow A$ which is convolution invertible. Similarly the unitality condition could be omitted.\\

  It is known that any  total cointegral  $f\colon  C\longrightarrow H$  which is  a coalgebra map is convolution invertible where the inverse is given by $f^{-1}=S\circ f$. Here $S$ is the antipode of the Hopf algebra $H$.
Therefore any Hopf Galois coextension with a total cointegral which is a coalgebra morphism is a cleft coextension.
In a special case when $S^2=\Id$ (for example when $H$ is commutative or cocommutative ) or when $C$ is cocommutative,  any total cointegral which is anti-coalgebra map is convolution invertible. One notes that the inverse map $f^{-1}$ is not a $H$-module map, but it satisfies

\begin{equation}\label{inverse-module}
   f^{-1}(ch)=  S(h) f^{-1}(c) .
\end{equation}
  This is because the maps $m\circ tw \circ(f^{-1} \ot S)$ and $f^{-1} \circ \gamma $ appeared in the right and left hand sides of \eqref{inverse-module} have the same two sided inverse $m\circ (f\ot \Id_H)$ in convolution algebra $\Hom(C\ot H, H)$. Here $m$ is the multiplication map of $H$ and $\gamma\colon  C\ot H\longrightarrow C$ is the right $H$-action on $C$ and $tw\colon  H\ot H\longrightarrow H\ot H$ is the twist map.
 Dually if the total integral $f: H\longrightarrow A$ is an algebra map then $f^{-1}:=f\circ S$ defines a convolution inverse of $f$.  Similarly the inverse map $f^{-1}:H\longrightarrow A$ is not a comodule map but it satisfies
 \begin{equation}
   f^{-1}(h)\ns{0}\ot f^{-1}(h)\ns{1}=f^{-1}(h\ps{2})\ot S(h\ps{1}).
 \end{equation}

\section{AYD and YD modules from total (co)integrals }
In this section, we study Yetter-Drinfeld and anti Yetter-Drinfeld modules constructed by total (co)integrals.
First, we explain how total (co)integrals which are (co)algebra maps can produce different types of anti Yetter-Drinfeld  and Yetter-Drinfeld modules. Then we show that those total (co)integrals  which are also (co)algebra maps and  they satisfy  certain properties yield stability condition and therefore stable anti Yetter-Drinfeld modules. More precisely for any total (co)integral which is a (co)algebra map we introduce an AYD and an YD module. Furthermore if the Hopf algebra $H$ has a modular pair in involution  then we produce two  different AYD and YD modules.
Then we replace the (co)algebra map property of the total (co)integral  by cleftness of Hopf Galois (co)extensions and similarly we construct all the four YD and AYD modules that we constructed in the previous case. We remind that anti Yetter-Drinfeld modules  are suitable coefficients for Hopf cyclic cohomology. Here we recall the definition of three  types anti Yetter-Drinfeld modules that will appear later in this paper.
The   module and comodule  $M$ over a Hopf algebra $H$ is called an anti-Yetter-Drinfeld module  \cite {hkrs1}  iff the action and coaction are compatible in the following sense;
\begin{align*}
  & \rho(hm)= h\ps{2}m\ns{0}\ot h\ps{3}m\ns{1}S(h\ps{1}), \quad ~~~~~~ \text{if M is a left module and a right comodule,}\\
  &\rho(mh)= S(h\ps{3})m\ns{-1}h\ps{1}\ot m\ns{0}h\ps{2}, \quad ~~~~~ \text{if M is a right module and a left comodule,}\\
  &\rho(mh)= m\ns{0}h\ps{2}\ot S^{-1}(h\ps{1})m\ns{1}h\ps{3}, \quad ~~~ \text{if M is a right module and a right comodule.}
\end{align*}
Furthermore the module $M$ is called stable if  $m\ns{1}m\ns{0}=m$, $m\ns{0}m\ns{-1}=m$, and $m\ns{0}m\ns{1}=m$ for the left-right, right-left, and right-right cases, respectively.
The Yetter-Drinfeld condition simply is obtained by replacing $S$ and $S^{-1}$ by each other in the anti Yetter-Drinfeld condition.

\subsection{(Anti) Yetter-Drinfeld modules constructed by total integrals}
In this subsection, we use total integrals to produce (stable) anti Yetter-Drinfeld and Yetter-Drinfeld modules.
 Let us first recall the following fact.
\begin{lemma}
  Let  $f\colon  H\longrightarrow A$ be a total integral which is an algebra map. Then  $A$ and the subalgebra $A^B=\{a\in A, ~ ba=ab\}$, the centralizer of $A$ in $B$, are both right $H$-modules by the action  given by
  \begin{equation}
    ah= f^{-1}(h\ps{1})a f(h\ps{2}).
  \end{equation}
\end{lemma}

 The following proposition  shows how special types of total integrals yield Yetter-Drinfeld modules.

\begin{proposition}\cite{cfm}\label{YD-extension}
 Let $H$ be a Hopf algebra,  $A$ be a right $H$-comodule algebra and  $f\colon H\longrightarrow A$ be a total integral which is an algebra map. Then $A$ is a right-right Yetter-Drinfeld module by the original coaction of $H$ and the following action;
 \begin{equation}
   ah= f^{-1}(h\ps{1})a f(h\ps{2}).
 \end{equation}
\end{proposition}

In the following proposition we introduce an  anti Yetter-Drinfeld module  constructed by special type of total integrals.

\begin{proposition}\label{lr-anti}
  Let $H$ be a Hopf algebra,  $A$ be a right $H$-comodule algebra and  $f\colon H\longrightarrow A$ be a total integral which is an algebra map. Then $A_B=A/[A,B]$    is a left-right  anti Yetter-Drinfeld module over $H$ by the original coaction of $H$ and the following action;
 \begin{equation}
   ha= f(h\ps{2})a f^{-1}(h\ps{1}).
 \end{equation}
 Furthermore if
 \begin{equation}\label{stable}
   a\ns{0}f^{-1}(a\ns{1}\ps{1}) f(a\ns{1}\ps{2})= a,
 \end{equation}
  then this action is stable.
\end{proposition}
\begin{proof}
  The action is associative by the algebra map property of $f$ and anti algebra map property of $f^{-1}$ and it is unital by unitality of $f$ and $f^{-1}$. The following computation shows the AYD condition.
  \begin{align*}
    &(ha)\ns{0}\ot (ha)\ns{1}\\
    &=f(h\ps{2})\ns{0}a\ns{0}f^{-1}(h\ps{1})\ns{0}\ot f(h\ps{2})\ns{1}a\ns{1}f^{-1}(h\ps{1})\ns{1}\\
    &=f(h\ps{2}\ps{1})a\ns{0}f^{-1}(h\ps{1}\ps{2})\ot h\ps{2}\ps{2}a\ns{1}S(h\ps{1}\ps{1})\\
    &=f(h\ps{3})a\ns{0}f^{-1}(h\ps{2})\ot h\ps{4}a\ns{1}S(h\ps{1})\\
    &=f(h\ps{2}\ps{2})a\ns{0}f^{-1}(h\ps{2}\ps{1})\ot h\ps{3}a\ns{1}S(h\ps{1})\\
    &=h\ps{2}a\ns{0}\ot h\ps{3}a\ns{1}S(h\ps{1}).
  \end{align*}
  The stability condition is obvious.
\end{proof}

By far from any total integral which is an algebra map we have an YD over $A$ \cite{cfm} and an AYD over $A_B=A/[A,B]$.
Now,  we construct different AYD and YD modules by a total integral. To do this, we need a condition on the Hopf algebra $H$. Let us recall the notion of modular pair in involution \cite{CM98}. Let $k$ be the ground field of a Hopf algebra $H$. Any unital algebra map  $\delta\colon  H\longrightarrow k$ is called a character. If $\sigma\in H$ is a group-like element \ie $\Delta(\sigma)=\sigma\ot \sigma$ , then the pair $(\delta, \sigma)$ is called a modular pair if $\delta(\sigma)=1$. Furthermore if $ \widetilde{S}^2(h)=\sigma h \sigma^{-1}$ where $\widetilde{S}(h)=\delta(h\ps{1}) S(h\ps{2})$ then $(\delta, \sigma)$ is called a modular pair in involution. We remind that the notion of modular pair in involution has important role to define Hopf cyclic cohomology \cite{CM98}.

\begin{proposition}\label{yd to ayd}
  Let $H$ be a Hopf algebra over a field $k$ with a modular pair in involution $(\delta, \sigma)$.
  If $f\colon  H^{cop}\longrightarrow A$ is a total integral which is an algebra map then $A\ot ^{\delta}k_{\sigma}$ is af right-right anti Yetter-Drinfeld module over $H^{cop}$ by the  action and coaction given by
  \begin{equation}
      (a\ot 1_k)h=f^{-1}(h\ps{1})af(h\ps{2}) \ot \delta(h\ps{3}), \quad ~~ a\ot 1_k\longmapsto a\ns{0} \ot 1_k \ot a\ns{1}\sigma.
  \end{equation}
Furthermore this action is stable if
\begin{equation}
   f^{-1}(\sigma) f^{-1}(a\ns{1}\ps{1})  a\ns{0}  f( a\ns{1}\ps{2})f(\sigma) \delta(a\ns{1}\ps{3})= a
\end{equation}
\end{proposition}
\begin{proof}
We note that $A$ is a right-right Yetter-Drinfeld module over $H^{cop}$ by Proposition \ref{YD-extension}. It is known \cite{hkrs1} that if $H$ has a modular pair in involution then the ground field $k$ is a right-left stable anti Yetter-Drinfeld module over $H$ by the action
$1h=\delta(h)$ and the coaction $1\longmapsto \sigma\ot 1$.  We denote this module by $^{\delta}k_{\sigma}$. If $H$ is a Hopf algebra with invertible antipode $S$ then $H^{cop}$ is a Hopf algebra with antipode $S^{-1}$. Since  $^{\delta}k_{\sigma}$ is a right-left anti Yetter-Drinfeld module over $H$ then it  is a right-right anti Yetter-Drinfeld module over $H^{cop}$ simply by twisting the left coaction to obtain a right coaction. Now we use the fact that the category of anti Yetter-Drinfeld modules is a monoidal category over the category of Yetter-Drinfeld modules. Precisely the tensor product of a right-right Yetter-Drinfeld module $A$ by a right-right anti Yetter-Drinfeld  module  $^{\delta}k_{\sigma}$  is a right-right anti Yetter-Drinfeld module over $H$.
\end{proof}

\begin{corollary}
  Let $H$ be a cocommutative Hopf algebra over a field $k$ with a modular pair in involution $(\delta, \sigma)$.
  If $f\colon  H\longrightarrow A$ is a total integral which is an algebra map then $A\ot ^{\delta}k_{\sigma}$ is an right-right anti Yetter-Drinfeld module over $H$.
\end{corollary}

In the following proposition we introduce a different Yetter-Drinfeld  module constructed by special total integrals.

\begin{proposition}
  Let $H^{op,cop}$ be a Hopf algebra with  a modular pair in involution $(\delta, \sigma)$,  and  $f\colon H\longrightarrow A$ be a total integral which is an algebra map. Then $A_B=A/[A,B]$    is a right-left   Yetter-Drinfeld module over $H^{op,cop}$ by the action and coaction given by
 \begin{equation}
   ah= f(h\ps{2})a f^{-1}(h\ps{1})\delta(S(h\ps{3})), \quad ~~~~~ a\longmapsto a\ns{0}\ot \sigma^{-1}a\ns{1}.
 \end{equation}
\end{proposition}
\begin{proof}
  This is the result of the fact that any left-right anti Yetter-Drinfeld module (e.g $A_B$ ) over a Hopf algebra $H$ can be turned in to a right-left anti Yetter-Drinfeld module over $H^{op,cop}$. Then the result follows from \cite[Theorem 2.1]{staic} as one can turn an anti Yetter-Drinfeld module to a Yetter-Drinfeld module.
\end{proof}

One notes that if $(\varepsilon, 1)$ is the modular pair in involution of $H$, (e.g group algebra and universal enveloping algebra), then the action and the coaction in the previous Proposition is the same as the ones in the Proposition \ref{lr-anti}.\\

By far we constructed four YD and AYD modules for total integrals which are algebra maps.
Now we aim to replace the algebra map property of the total integral $f$ by  cleftness of a  Hopf Galois extension.
   Let $A(B)^H$ be a Hopf Galois extension. In this case the map
   \begin{equation}
     \kappa\colon  H\longrightarrow (A\ot_B A)^B, \quad \kappa(h)= can^{-1}(1_A\ot h),
   \end{equation}
      is an anti-algebra map \cite{JS}. Here
      $$(A\ot_B A)^B= \{a\ot a' \in A \ot_B A, ~ ba\ot a'=a\ot a'b\}.$$ The algebra structure of $(A\ot_B A)^B$ is given by
  \begin{equation}
    (a_1\ot a_1')(a_2\ot a_2')=a_1a_2\ot a_2'a_2.
  \end{equation}
  We denote $\kappa(h)= \kappa^1(h)\ot \kappa^2(h)$. One notes that the anti-algebra map property of $\kappa$ is equivalent to
  \begin{equation}
    \kappa^1(hk)\ot \kappa^2(hk)=\kappa^1(k)\kappa^1(h)\ot \kappa^2(h)\kappa^2(k),
  \end{equation}
  for all $h,k\in H$. The following lemma enables us to find out that   although the  total integrals of a cleft Hopf Galois extension are not algebra maps in general, they have close relation to an algebra map.

\begin{lemma}\label{inverse-kappa}
  Let $A(B)^H$ be a cleft Hopf Galois extension with  a convolution invertible total integral $f\colon  H\longrightarrow A$. Then

  \begin{equation}
    \kappa(h)= f^{-1}(h\ps{1})\ot f(h\ps{2})
  \end{equation}

\end{lemma}
\begin{proof}
 This is followed by  $can^{-1}(a\ot h)=af^{-1}(h\ps{1})\ot f(h\ps{2})$.
\end{proof}
\begin{remark}
  {\rm
  Let $A(B)^H$ be a cleft Hopf Galois extension with total integral $f$. The previous lemma  shows that although the total integral $f$ is not an algebra map in general it satisfies the following relation;

  \begin{equation}
    f^{-1}(h\ps{1}k\ps{1})\ot f(h\ps{2}k\ps{2})= f^{-1}(k\ps{1})f^{-1}(h\ps{1})\ot f(h\ps{2})f(k\ps{2}).
  \end{equation}

  }
\end{remark}

Now we are ready to replace the algebra map property of a total integral $f\colon H\longrightarrow A$, discussed before,  by the condition  that $A(B)^H$ is a cleft Hopf Galois extension. In fact the associativity of the  $H$-action over $A$  which was the result of
 the algebra map property of the total integral, now can be obtained  by the anti algebra map property of $\kappa$.
\begin{lemma}
  Let $A(B)^H$ be a cleft Hopf Galois extension by the total integral $f\colon  H\longrightarrow A$. Then $A$ and  $A^B=\{a\in A, ~ ba=ab\}$, are both right $H$-modules by the action  given by
  \begin{equation}
    ah= f^{-1}(h\ps{1})a f(h\ps{2}).
  \end{equation}
\end{lemma}
\begin{proof}
This is due to the fact that  $\kappa$ is an anti algebra map and therefore   the  right $H$-action $ ah=\kappa^{1}(h)a\kappa^2(h)$ turns $A$ and $A^B$ to an associative and unital action.
\end{proof}
The following proposition introduces a new YD module for any Hopf Galois extension.
\begin{proposition}
   Let   $A(B)^H$ be a  Hopf Galois extension.  Then $A$ is a right-right Yetter-Drinfeld module over $H$ by the associative and unital action given by
  \begin{equation}
    ah= \kappa^1(h)a \kappa^2(h).
  \end{equation}
\end{proposition}
\begin{proof}
The associativity of the action is the result of the anti algebra map property of $\kappa$. The unitality of the action is obvious by unitality of $\kappa$. For the right-right YD condition we notice that the map $\kappa$ has the following property \cite{JS};
\begin{equation}\label{nice relation}
  \kappa^1(h)\ns{0}\ot \kappa^2(h)\ns{0} \ot \kappa^1(h)\ns{1}\ot \kappa^2(h)\ns{1}=\kappa^1(h\ps{2})\ot \kappa^2(h\ps{2})\ot S(h\ps{1})\ot h\ps{3}.
\end{equation}
The following computation proves the YD condition.
\begin{align*}
  &(\kappa^1(h)a \kappa^2(h))\ns{0}\ot (\kappa^1(h)a \kappa^2(h))\ns{1}\\
  &=\kappa^1(h)\ns{0} a\ns{0}\kappa^2(h)\ns{0}\ot \kappa^1(h)\ns{1} a\ns{1}\kappa^2(h)\ns{1}\\
  &=\kappa^1(h\ps{2})a\ns{0}\kappa^2(h\ps{2})\ot S(h\ps{1})a\ns{1}h\ps{3}\\
  &=a\ns{0}h\ps{2}\ot S(h\ps{1})a\ns{1}h\ps{3}.
\end{align*}
We used \eqref{nice relation} in the second equality.
\end{proof}

In the following proposition we introduce an YD module structure  over $A$ for cleft Hopf Galois extensions.

\begin{proposition}
  Let   $A(B)^H$ be a cleft Hopf Galois extension with the total integral $f\colon  H\longrightarrow A$. Then $A$ is a right-right Yetter-Drinfeld module over $H$ by the associative and unital action given by
  \begin{equation}
    ah= f^{-1}(h\ps{1})a f(h\ps{2}).
  \end{equation}
  \end{proposition}
\begin{proof}
  This is the immediate result of the previous lemma and $ f^{-1}(h\ps{1})\ot f(h\ps{2})=\kappa^1(h)\ot  \kappa^2(h)$.
\end{proof}
In the following proposition we introduce an SAYD module over $A$ for cleft Hopf Galois extensions.
\begin{proposition}
  Let  $A(B)^H$ be a cleft Hopf Galois extension by the total integral $f\colon  H\longrightarrow A$. Then $A_B= A/[A,B]$ is a left-right \emph{stable} anti Yetter-Drinfeld module by the action given by

  \begin{equation}
    ha=f(h\ps{2}) a f^{-1}(h\ps{1}).
  \end{equation}

\end{proposition}
\begin{proof}
  By \cite{JS} the quieten space $A_B$ is a left-right stable anti Yetter-Drinfeld module by the left $H$-action given by
  \begin{equation}
    ha=\kappa^2(h) a \kappa^1(h).
  \end{equation}
  Now the statement is the result of Lemma \ref{inverse-kappa}.
\end{proof}
One notes that the stability condition in the previous proposition is obtained by the fact that $A(B)^H$ is a Hopf Galois extension. As we noticed before the stability condition can not be obtained automatically  by total integrals which are algebra maps. But if $A(B)^H$ is a cleft Hopf Galois extension then not only we obtain an AYD module using the total integral but also we obtain the stability condition.
Now for the Hopf algebras with a modular pair in involution we introduce  different AYD and YD modules.

\begin{proposition}
  Let $H$ be a Hopf algebra over a field $k$ with a modular pair in involution
$(\delta,\sigma)$ and     let $A(B)^{H^{cop}}$ be a cleft Hopf Galois extension with the total integral $f: H^{cop}\longrightarrow A$.  Then $A\ot {}^{\delta}k_{\sigma}$ is a  right-right \emph{stable} anti Yetter-Drinfeld module over $H^{cop}$ by the coaction and action given by
  \begin{equation}
    ah= f^{-1}(h\ps{1})a f(h\ps{2})\ot \delta({h\ps{3}}), \quad a\ot 1_k\longmapsto a\ns{0}\ot 1_k\ot a\ns{1}\sigma.
  \end{equation}

\end{proposition}

\begin{proposition} Let $H^{op,cop}$ be a Hopf algebra with  a modular pair in involution $(\delta, \sigma)$ and let   $A(B)^{H^{op,cop}}$ be a cleft Hopf Galois extension with the total integral $f\colon  H^{op,cop}\longrightarrow A$. Then
  $A_B$ is a right-right  Yetter-Drinfeld  module over $H^{op,cop}$ with the following action and coaction,
  $$f(h\ps{2})af^{-1}(h\ps{1})\delta(S(h\ps{3}), \quad ~~ a\longmapsto a\ns{0}\ot \sigma^{-1} a\ns{1}.$$
\end{proposition}

\begin{corollary}
  Let $H$ be a cocommutative Hopf algebra with  a modular pair in involution $(\delta, \sigma)$ and  let  $A(B)^{H}$ be a cleft Hopf Galois extension with the total integral $f\colon  H\longrightarrow A$. Then
  $A\ot ^{\delta}k_{\sigma}$ is an right-right stable anti Yetter-Drinfeld  module over $H$.
\end{corollary}

Since $can$ is a right $H$-comodule morphism the map $\kappa$ is a right $H$-comodule map. If $H$ is a commutative Hopf algebra then $(A\ot_B A)^B$ is a right $H$-comodule algebra by the coaction given by $a\ot_B a'\longmapsto a \ot a'\ns{0}\ot_B a'\ns{1}$ which is well-defined because the right $H$-coaction is $B$-linear. In this case $\kappa$ is an algebra map which is a $H$-comodule map and therefore it is a total integral.
The following proposition states that any commutative Hopf Galois extension provides a convolution invertible total integral.
\begin{proposition}
  Let $H$ be a commutative Hopf algebra and $A(B)^H$ be a Hopf Galois extension. Then the map $\kappa: H\longrightarrow (A\ot_B A)^B$ is a convolution invertible total integral.
\end{proposition}

\begin{proposition}
  Let $H$ be a commutative Hopf algebra    and $A(B)^H$ be a Hopf Galois extension. Then the following Galois map
  \begin{equation}
    (A\ot_B A)^B\ot (A\ot_B A)^B\longrightarrow (A\ot_B A)^B \ot H,
  \end{equation}
  given by
  \begin{equation}
    (x\ot y) \ot (x'\ot y')\longmapsto xy\ot y'\ns{0}y\ot y\ns{1},
  \end{equation}
  is surjective.
\end{proposition}
One notes that the Galois map in the previous proposition is not necessarily injective in general.


\subsection{(Anti) Yetter-Drinfeld modules constructed by total cointegrals}
In this subsection we explain the dual results of the previous subsection for total cointegrals.
Let $C$ be a right $H$-module, $D=C/I$ and $\pi\colon C\longrightarrow D$. By \cite{hassanzadeh1} we set
\begin{equation}
  C^D=\left\{c\in C, \quad  c\ps{1}\varphi(\pi(c\ps{2}))= c\ps{2}\varphi(\pi(c\ps{1})) \right \}_{\varphi\in D^*},
\end{equation}
where $D^*$ is the algebraic dual of $D$.
We define
\begin{equation}
  C_D=\frac{C}{W},
\end{equation}
where
$$W= \left\{c\ps{1}\varphi(c\ps{2})- c\ps{2}\varphi(c\ps{1}), \quad c\in C \right \}_{\varphi\in D^*}.$$

\begin{lemma}
  Let  $f\colon  C\longrightarrow H$ be a total cointegral which is an coalgebra map. Then  $C$ and  $C^D$,  are both right $H$-comodules by the coaction  given by
  \begin{equation}
    c\longmapsto c\ps{2}\ot f^{-1}(c\ps{1})f(c\ps{3}).
  \end{equation}
\end{lemma}
\begin{proof}
  The coassociativity of the coaction is the result of coalgebra and anti-coalgebra map properties of $f$ and $f^{-1}$.

\end{proof}
The following proposition explains how a special type of total cointegrals produce YD modules.

\begin{proposition}\label{YD-coextension}
 Let   $f\colon C\longrightarrow H$ be a  total cointegral which is a coalgebra map. Then $C$ is a right-right Yetter-Drinfeld module by the original action of $H$ and the following coaction;
 \begin{equation}
    c\longmapsto c\ps{2}\ot f^{-1}(c\ps{1})f(c\ps{3}).
 \end{equation}
\end{proposition}
\begin{proof} The following computation shows the Yetter-Drinfeld condition.
 \begin{align*}
    &\rho(ch)= c\ps{2}h\ps{2}\ot f^{-1}(c\ps{1}h\ps{1})f(c\ps{3}h\ps{3})\\
    &= c\ps{2}h\ps{2}\ot S(h\ps{1}) f^{-1}(c\ps{1})f(c\ps{3})h\ps{3}\\
    &=c\ns{0}h\ps{2}\ot S(h\ps{1})c\ns{1}h\ps{3}.
  \end{align*}
\end{proof}
Here we introduce an AYD module which is constructed by total cointegrals.
\begin{proposition}
  Let $H$ be a Hopf algebra   and  $f\colon H\longrightarrow A$ be a total integral which is a coalgebra map. Then $C^D$    is a right-left  anti Yetter-Drinfeld module over $H$ by the original action of $H$ and the following coaction;
 \begin{equation}
   c\longmapsto f^{-1}(c\ps{3})f(c\ps{1})\ot c\ps{2}.
 \end{equation}
 Furthermore this action is stable if
 \begin{equation}
      c= c\ps{2}f^{-1}(c\ps{3})f(c\ps{1}).
 \end{equation}

\end{proposition}

\begin{proof}
  The following computation shows the anti Yetter-Drinfeld condition.
  \begin{align*}
 & (ch)\ns{-1}\ot (ch)\ns{0}\\
   & =f^{-1}(c\ps{3}h\ps{3})f(c\ps{1}h\ps{1})\ot c\ps{2}h\ps{2}\\
   &=S(h\ps{3})c\ps{3}f(c\ps{1})h\ps{1}\ot c\ps{2}h\ps{2}\\
   &=S(h\ps{3})c\ns{-1}h\ps{1}\ot c\ns{0}h\ps{2}.
  \end{align*}
\end{proof}

Now for the Hopf algebras which admit a modular pair in involution we introduce different YD and AYD modules.

\begin{proposition}\label{dual-yd to ayd}
  Let $H$ be a Hopf algebra over a field $k$ with a modular pair in involution $(\delta, \sigma)$.
  If $f\colon  C \longrightarrow H^{cop}$ is a total cointegral which is a coalgebra map then $C\ot ^{\delta}k_{\sigma}$ is an right-right anti Yetter-Drinfeld module over $H^{cop}$ by the  action and coaction given by
  \begin{equation}
      (c\ot 1_k)h=ch\ps{1}\ot \delta(h\ps{2}), \quad ~~ c\ot 1_k\longmapsto c\ps{2} \ot 1_k \ot f^{-1}(c\ps{1})f(c\ps{3})\sigma.
  \end{equation}
Furthermore this action is stable if
\begin{equation}
  c\ps{3} f^{-1}(c\ps{2}) f(c\ps{4}) \sigma \ot \delta(  f^{-1}(c\ps{1})f(c\ps{5})= c\ot 1_k.
\end{equation}
\end{proposition}
\begin{proof}
  The proof is similar to Proposition \ref{dual-yd to ayd}.
\end{proof}

\begin{corollary}
  Let $H$ be a cocommutative Hopf algebra over a field $k$ with a modular pair in involution $(\delta, \sigma)$.
  If $f\colon  C\longrightarrow H$ is a total cointegral which is a coalgebra map then $C\ot ^{\delta}k_{\sigma}$ is an right-right anti Yetter-Drinfeld module over $H$.
\end{corollary}

\begin{remark}{\rm
Let $H^{op,cop}$ be a Hopf algebra over a field $k$ with a modular pair in involution $(\delta, \sigma)$ and
   $f\colon  C \longrightarrow H^{op, cop}$ be a total cointegral which is a coalgebra map. Then by \cite{staic} the right-left AYD module $C^D$    over $H^{op, cop}$ can be turned in to  a right-left YD module over $H^{op,cop}$ by twisting the action and coaction using $\delta$ and $\sigma$.

}

\end{remark}

Now we  replace the coalgebra map property of the total cointegral $f$ by a cleft Hopf
Galois coextension. For any Hopf Galois coextension $C(D)_H$, by \cite{hassanzadeh1} we set
\begin{equation}
  (C\Box_D C)_D= \frac{C\Co_D C}{W},
\end{equation}
where
\begin{equation}\label{ww}
  W= \langle   c\ot c'\ps{1} \varphi(\pi(c'\ps{2}))- c\ps{2}\ot c'\varphi(\pi(c\ps{1}))  \rangle,
\end{equation}
 and  $\varphi\in D^*$, $c\ot c'\in C\Co_D C$. We denote the elements of the quotient by an over line.  This quotient space   $(C\Co_D C)_D$  is a coassociative counital coalgebra \cite{hassanzadeh1} where the coproduct and counit maps are given by
\begin{equation}\label{coproduct-main}
  \Delta(\overline{c\ot c'})= \overline{c\ps{1}\Co_D c'\ps{2}}\ot \overline{c\ps{2}\Co_D  c'\ps{1}}. \qquad \ve(\overline{c\ot c'})=\ve(c)\ve(c').
\end{equation}
Since for any Hopf Galois coextension $C(D)_H$ we have $\pi(ch)=\varepsilon(h)\pi(c)$ then $(C\Box_D C)_D$ is a right $H$-module by the right action given by
\begin{equation}
  (c\ot c')h= c\ot c'h,~~~ c,c'\in C , h\in H.
\end{equation}
By \cite{hassanzadeh1} the map

\begin{equation}
  \kappa\colon =(\varepsilon\ot \Id_H)\circ {can}^{-1}\colon (C\Box_DC)_D\longrightarrow H
\end{equation}

 is an anti-coalgebra map which is equivalent to

 \begin{equation}
  \kappa(\overline{c\ot c'})\ps{1}\ot \kappa(\overline{c\ot c'})\ps{2}=\kappa(\overline{c\ps{2}\ot c'\ps{1}})\ot \kappa(\overline{c\ps{1}\ot c'\ps{2}}).
\end{equation}
 The following lemma enables us to find out that although the total
cointegrals of a cleft Hopf Galois coextension are not coalgebra maps in general, they have a close relation to
a coalgebra map.
 \begin{lemma}\label{inverse-co}
   Let $C(D)_H$ be a cleft Hopf Galois coextension with convolution invertible total cointegral $f\colon C\longrightarrow H$. Then
   \begin{equation}
     \kappa(c\ot c')=f^{-1}(c) f(c'), \quad ~~~~~~ c\ot c'\in (C\Box_D C)_D.
   \end{equation}
 \end{lemma}
 \begin{proof}
   It is enough to show that
   \begin{equation}\label{can-cointegral}
     can^{-1}(c\ot c')= c\ps{1}\ot  f^{-1}(c\ps{2})f(c').
   \end{equation}
   To prove this, we observe that since $can $ is bijective canonical map then there exists $\sum_ic_i\ot h_i\in C\ot H$ such that $can(\sum_ic_i\ot h_i)=c\ot c'$.
   Now, we have
   \begin{align*}
     &\sum_ican^{-1}(can(c_i\ot h_i))= \sum_ican^{-1}(c_i\ps{1}\ot c_i\ps{2}h_i)\\
     &=\sum_i c_i\ps{1}\ot  f^{-1}(c_i\ps{2})f(c_i\ps{3}h_i)=\sum_i c_i\ps{1}\ot  f^{-1}(c_i\ps{2})f(c_i\ps{3})h_i\\
     &=\sum_ic_i\ps{1}\ot \varepsilon(c_i\ps{2})h_i=\sum_ic_i\ot h_i.
   \end{align*}
   This shows that the equation \eqref{can-cointegral} is the correct formula for $can^{-1}$.
 \end{proof}
Let $C(D)_H$ be a cleft Hopf Galois coextension with a total cointegral $f\colon C\longrightarrow H$. The
previous lemma  and the anti coalgebra map property of $\kappa$ show that although the total cointegral $f$ is not a coalgebra map in general
it satisfies the following property;

\begin{equation}
  f^{-1}(c)\ps{1}f(c')\ps{1}\ot f^{-1}(c)\ps{2}f(c')\ps{2} = f^{-1}(c\ps{2})f(c'\ps{1})\ot f^{-1}(c\ps{1})f(c'\ps{2}).
\end{equation}

\begin{lemma}\label{lem-comod}
  Let $C(D)_H$ be a  Hopf Galois coextension. Then $C$ is a right $H$-comodule by the coaction  given by
  \begin{equation}
    c\longmapsto c\ps{2}\ot \kappa(c\ps{1}\ot c\ps{3}).
  \end{equation}
\end{lemma}
\begin{proof}
  The following computation shows the coassociativity of the coaction.
  \begin{align*}
    &c\ns{0}\ot c\ns{1}\ps{1}\ot c\ns{1}\ps{2}\\
    &=c\ps{2}\ot \kappa(c\ps{1}\ot c\ps{3})\ps{1}\ot \kappa(c\ps{1}\ot c\ps{3})\ps{2}\\
    &=c\ps{2}\ot \kappa(c\ps{1}\ps{2}\ot c\ps{3}\ps{1})\ot \kappa(c\ps{1}\ps{1}\ot c\ps{3}\ps{2})\\
    &=c\ps{3}\ot \kappa(c\ps{2}\ot c\ps{4})\ot \kappa(c\ps{1}\ot c\ps{5})\\
    &=c\ps{2}\ps{2}\ot \kappa(c\ps{2}\ps{1}\ot c\ps{2}\ps{3})\ot \kappa(c\ps{1}\ot c\ps{3})\\
    &=c\ns{0}\ns{0}\ot c\ns{0}\ns{1}\ot c\ns{1}.
  \end{align*}
  We used the coalgebra map property of $\kappa$ in the second equality.
\end{proof}

\begin{lemma}
  Let $C(D)_H$ be a cleft Hopf Galois coextension by the total cointegral $f\colon  C\longrightarrow H$. Then $C$ is a right $H$-comodule by the coaction  given by
  \begin{equation}
    c\longmapsto   c\ps{2}\ot f^{-1}(c\ps{1})f(c\ps{3}).
  \end{equation}
\end{lemma}
\begin{proof}
  This is the result of the Lemmas \ref{inverse-co}  and \ref{lem-comod}.
\end{proof}
\begin{proposition}
   Let   $C(D)_H$ be a  Hopf Galois coextension.  Then $C$ is a right-right Yetter-Drinfeld module over $H$ by the original action of $H$ and the following  coaction
  \begin{equation}
   c\longmapsto   c\ps{2}\ot f^{-1}(c\ps{1})f(c\ps{3}).
  \end{equation}
\end{proposition}
\begin{proof}
  This is the result of Lemmas \ref{inverse-co}  and \ref{lem-comod} and the Proposition \ref{YD-coextension}.
\end{proof}
Here we introduce a \emph{stable} anti Yetter-Drinfeld module constructed by  total cointegrals of cleft Hopf Galois coextensions. The stability condition which is not followed immediately by the total cointegrals which are coalgebra maps, satisfies here for free.
\begin{proposition}
  Let  $C(D)_H$ be a cleft Hopf Galois coextension by the total cointegral $f\colon  C\longrightarrow H$. Then $C^D$ is a right-left stable anti Yetter-Drinfeld module by the coaction given by

  \begin{equation}
    c\longmapsto c\ps{2}\ot f^{-1}(c\ps{1})f( c\ps{3}).
  \end{equation}

\end{proposition}
\begin{proof}
  By \cite{hassanzadeh1} the subspace  $C^D$ is a right-left \emph{stable} anti Yetter-Drinfeld module by the left $H$-coaction given by
  \begin{equation}
    c\longmapsto c\ps{2}\ot \kappa(c\ps{1}\ot c\ps{3}).
  \end{equation}
  Now the statement is the result of Lemma \ref{inverse-co}.
\end{proof}

Now we introduce different YD and AYD modules when the Hopf algebra of the Hopf Galois coextension has a modular pair in involution.
\begin{proposition}
  Let $H$ be a Hopf algebra over a field $k$ with a modular pair in involution $(\delta, \sigma)$, and $C(D)_{H^{cop}}$ be a cleft Hopf Galois coextension with a total cointegral $f\colon  C \longrightarrow H^{cop}$. Then $C\ot ^{\delta}k_{\sigma}$ is a right-right stable anti Yetter-Drinfeld module over $H^{cop}$ by the  action and coaction given by
  \begin{equation}
      (c\ot 1_k)h=ch\ps{1}\ot \delta(h\ps{2}), \quad ~~ c\ot 1_k\longmapsto c\ps{2} \ot 1_k \ot f^{-1}(c\ps{1})f(c\ps{3})\sigma.
  \end{equation}

\end{proposition}

\begin{corollary}
  Let $H$ be a cocommutative Hopf algebra over a field $k$ with a modular pair in involution $(\delta, \sigma)$, and $C(D)_H$ be a cleft Hopf Galois coextension with a total cointegral $f\colon  C \longrightarrow H$. Then $C\ot ^{\delta}k_{\sigma}$ is an right-right stable anti Yetter-Drinfeld module over $H$ by the  action and coaction given in the previous proposition.

\end{corollary}

One notes that again by \cite{staic} the right-left anti  Yetter-Drinfeld module $C^D$    over $H^{op,cop}$ can be turned in to  a right-left Yetter-Drinfeld module over $H^{op,cop}$.\\

If $H$ is a cocommutative Hopf algebra   then the original  action of $H$ over $C$ turns $(C\Box_D C)_D$ in to a right $H$-module coalgebra.
By \cite[Lemma 3.5]{hassanzadeh1} the map $\kappa$ is a right $H$-module map and therefore it is a total cointegral.
 Therefore we obtain the following lemma.

\begin{lemma}\label{lemmakappa}
  Let $H$ be a cocommutative Hopf algebra and $C(D)_H$ be a Hopf Galois coextension. Then  the space $(C\Box_D C)_D$ is a right $H$-module coalgebra and  the map $\kappa: (C\Box_D C)_D\longrightarrow H$ is a convolution invertible total cointegral.
\end{lemma}

Also the Galois map
\begin{equation}
  can\colon  (C\Box_DC)_D \ot H\longrightarrow (C\Box_DC)_D \Box (C\Box_DC)_D,
\end{equation}
given by
\begin{equation}
  (c\ot c')\ot h\longmapsto c\ps{1}\ot c'\ps{2}\ot c\ps{2}\ot c'\ps{1}h,
\end{equation}
is injective. This is because $can(c\ot 1_H)$ and $can'\colon  c'\ot h\longmapsto c'\ps{1}h\ot c'\ps{2}$ \cite[Lemma 4.1]{canapel} are both injective. This  map is not necessarily surjective in general. Therefore  we have the following proposition.
\begin{proposition}
 Let $H$ be a cocommutative Hopf algebra  and $C(D)_H$ be a Hopf Galois coextension. Then the following Galois map is injective.
 \begin{equation}
  can\colon  (C\Box_DC)_D \ot H\longrightarrow (C\Box_DC)_D \Box (C\Box_DC)_D.
\end{equation}

\end{proposition}


  \section{Examples}

In this section we introduce some examples of our results in the previous sections. Specially we apply them  to the  Connes-Moscovici Hopf algebra, universal enveloping algebras of  a Lie algebra and  polynomial algebras. We use our results  to classify (cleft) Hopf Galois (co)extensions and  total (co)integrals.

\begin{proposition}
 Let $H_n$ be the Connes-Moscovici Hopf algebra and $A$ be a finite dimensional comodule algebra over $H_n$ with non-trivial coaction. Then there is no total integral $f\colon  H_n\longrightarrow A$ which is an algebra map.
\end{proposition}
\begin{proof}
By the  results of the previous section $A$ is an YD module by the action $ah=f^{-1}(h\ps{1})af(h\ps{2})$ and the original non-trivial coaction.
  But this is in contradiction with the fact that  any finite dimensional YD module over the Connes-Moscovici Hopf algebra has trivial coaction \cite{RS}.
\end{proof}

Dually we have the following proposition.
\begin{proposition}
  Let $H_n$ be the Connes-Moscovici Hopf algebra and $C$ be a finite dimensional module coalgebra over $H_n$ with non-trivial action. Then there is no total cointegral $f\colon   C\longrightarrow H_n$ which is a coalgebra map.
\end{proposition}
\begin{proof}
  By the  results of the previous section $C$ is an YD module by the  the original non-trivial action.
  But this is in contradiction with the fact that  any finite dimensional YD module over the Connes-Moscovici Hopf algebra has trivial action \cite{RS}.
\end{proof}

\begin{proposition}
 Let $A$ be a finite dimensional comodule algebra over the Connes-Moscovici Hopf algebra $H_n$. Then there is no  right  Hopf Galois extension $A(B)^{H_n}$.
\end{proposition}
\begin{proof}
   It is known \cite{RS}  that   the Connes-Moscovici Hopf algebra admits only one finite dimensional stable anti Yetter-Drinfeld  module which is the ground field $\mathbb{C}$ by the right action and left coaction defined by  the modular pair  in involution $(\delta,1)$  \cite{CM98}. One notes that the only SAYD on the ground filed is of the type right-left. But we showed that  $A_B$ is a left-right stable anti Yetter-Drinfeld module  which is a contradiction.

\end{proof}

\begin{proposition}
Let $H_n$ be the Connes-Moscovici Hopf algebra and $ C$ be a cocommutative $H_n$-module coalgebra with non-trivial $H_n$-coaction. Then there is no finite dimensional right  Hopf Galois coextension $C(D)_{H_n}$.
\end{proposition}
\begin{proof}
  By the  results of the previous section $C^D$ is a right-left SAYD module over $H_n$. Since $C$ is cocommutative then $C^D=C$.
  Again by \cite{RS}     the Connes-Moscovici Hopf algebra admits only one finite dimensional stable anti Yetter-Drinfeld  module which is the ground field $\mathbb{C}$ by the right action and left coaction defined by  the modular pair  in involution $(\delta,1)$  \cite{CM98}. Therefore   $C^D\cong \mathbb{C}$ as a $H$-module-comodule. But this implies that the action of $H_n$ on $C^D=C$ should be trivial by $\delta$ which is a contradiction.
\end{proof}



\begin{proposition}
  Let $H$ be a Hopf algebra and $C$ be a $H$-module coalgebra. If there exists a coaction of $H$ that turns  $D$ into an $H$-comodule coalgebra such that $C\cong D  \# H$ as $H$-module coalgebras then

\begin{enumerate}[i)]

\item $C$ is a right-right Yetter-Drinfeld module over $H$.

\item $C^D$ is a right-left stable anti Yetter-Drinfeld module over $H$.

\end{enumerate}

\end{proposition}
\begin{proof}
It was shown in \cite{dmr} that the Hopf Galois coextension $C(D)_H$ is cleft if and only if $C\cong D \#_{\sigma} H$ as $H$-module coalgebras   where $ D \#_{\sigma} H$ is the crossed coproduct with respect to cocycle $\sigma\colon  C\longrightarrow H\ot H$.
 Then both (i) and (ii) are followed by our results in the previous section.
\end{proof}

Dually we have the following result.
\begin{proposition}
  Let $H$ be a Hopf algebra and $A$ be a $H$-comodule algebra. If there exists an action of $H$ that turns $B$ into an $H$-module algebra such that $A\cong B  \# H$ as $H$-comodule algebras then

\begin{enumerate}[i)]

\item $A$ is a right-right Yetter-Drinfeld module over $H$.

\item $A_B$ is a left-right stable anti Yetter-Drinfeld module over $H$.

\end{enumerate}

\end{proposition}
\begin{proof}
  This is the result of the fact that by \cite{DT} and  \cite{BM}  the Hopf Galois extension $A(B)^H$ is cleft if and only if $A\cong B \#_{\sigma} H$ as $H$-comodule algebras.
\end{proof}



\begin{proposition}
  Let $\mathfrak{g}$ be a Lie algebra over a field $k$ and $A$ be a $U(\mathfrak{g})$-comodule algebra. If there exists a map $\lambda\colon  \mathfrak{g}\longrightarrow A$ such that $\rho(\lambda(x))=\lambda(x)\ot 1+1\ot x$ and $\lambda([x,y])=\lambda(x)\lambda(y)-\lambda(y)\lambda(x)$ for all $x,y\in \mathfrak{g}$. Then
\begin{enumerate}[i)]

\item $A$ is a right-right Yetter-Drinfeld module over $U(\mathfrak{g})$.


\item $A_B$ is a left-right  anti Yetter-Drinfeld module over $U(\mathfrak{g})$.
\end{enumerate}
\end{proposition}
\begin{proof}
  By \cite{bell} the map $\lambda$ can be extended to a total integral $\lambda\colon U(\mathfrak{g})\longrightarrow A$ which is an algebra map. Then both  i) and ii)  are followed by our results in the previous section.
\end{proof}

\begin{example}{\rm
  Let $p$ be a prime number, $R$ be a commutative ring of characteristic $p$,  $G$ be cyclic group of order $p$ generated by the element $g$, $H=RG$ be the group  algebra of $G$ and $C=M_{p\times p}$ be the space of $p\times p$ matrices with entries in $R$. If  $A_{ij}$ is the standard basis for $C$ then  space $C$ is a coalgebra by $\Delta(A_{ij})= \sum_{k=0}^{p-1} A_{ik}\ot A_{kj}$, $\varepsilon(A_{ij})=\delta_{ij}$. The Hopf algebra $RG$ acts on $M_{p\times p}$ by
  \begin{equation}\label{action-ex1}
    A_{ij} g^n= A_{i+n(mod p)~ j+n(mod p)}.
  \end{equation}
    By this action $M_{p\times p}$ is a $RG$ module coalgebra. We define a total cointegral as follows,
  \begin{equation}
    f\colon  M_{p\times p}\longrightarrow RG, \quad  f(A_{ij})=\delta_{ij} g^i.
  \end{equation}
     It is easy to show that $f$ is a coalgebra map and therefore any Hopf Galois coextension $C( D)_H$ is  cleft.
     In this case, $C=M_{p\times p}$ is a right-right Yetter Drinfeld module by the action defined in \eqref{action-ex1} and the following coaction
     \begin{equation}
       A_{ij}\longmapsto \sum_{k=0}^{p-1}   \sum_{m=0}^{p-1} \delta_{ij}\delta_{kj}  A_{mk}\ot g^{k-i}.
     \end{equation}

Furthermore $ C^D$ is a stable anti Yetter-Drinfeld module over $RG$ by the action \eqref{action-ex1} and the following coaction

\begin{equation}
  A_{ij}\longmapsto \sum_{k=0}^{p-1}  \sum_{i=0}^{p-1} \delta_{kj}\delta_{im}~ g^{i-k}\ot A_{mk}.
\end{equation}

  }
\end{example}

Now we aim to explain some examples for polynomial algebras $R[x]$ as a special case of universal enveloping algebra of a one dimensional Lie algebra. In fact  $\Delta(x^n)= \Sigma_{i=0}^n {n \choose i}x^i \ot x^{n-i}$ and  $\varepsilon(x^n)=0$ for $n>0$.

\begin{example}{\rm
  Let $p$ be a prime number and  $R$ be a commutative unital ring of  characteristic $p$. We consider the polynomial algebra  $C=R[x]$ and its sub-Hopf algebra $H=R[x^p]$. Therefore $C$ is $H$-module coalgebra by multiplication. We define the total cointegral $f\colon  C\longrightarrow H$ given by

  \begin{equation}
 f(x^n)=\begin{cases}
x^n   \quad  \text{if} \quad&\text{if $p\mid n$}\\
0  \quad  \text{if} \quad   &\text{otherwise}.\\
\end{cases}
\end{equation}
 Since $R$ is of characteristics $p$ then it can be easily checked that $f$ is a coalgebra map and therefore  any Hopf Galois coextension $C(D)_H$ is  cleft. 
  Therefore $C=R[x]$ is a right-right Yetter-Drinfeld module over $R[x^p]$  by the multiplication as the action and the  following coaction,
 \begin{equation}
   x^n\longmapsto\begin{cases}
\sum_{i=j}^n\sum_{j=0}^i   {n\choose i}  {i\choose j}  x^{i-j}\ot x^{n+j-i}   \quad  \text{if} \quad&\text{if $p\mid n$}\\
0  \quad  \text{if} \quad   &\text{otherwise}.\\
\end{cases}
 \end{equation}
  Also $R[x]$ is a right-left  anti Yetter-Drinfeld module over $R[x^p]$ by the multiplication as the action and the following coaction,

\begin{equation}
   x^n\longmapsto\begin{cases}
\sum_{i=j}^n\sum_{j=0}^i   {n\choose i}  {i\choose j}  x^{i+j-n}\ot x^{i-j}   \quad  \text{if} \quad&\text{if $p\mid n$}\\
0  \quad  \text{if} \quad   &\text{otherwise}.\\
\end{cases}
 \end{equation}

 }
\end{example}

\begin{example}{\rm
 Let $R$ be a commutative ring and $H=R[x]$  with $\Delta(x^n)= \Sigma_{i=0}^n {n \choose i}x^i \ot x^{n-i}$, $\varepsilon(x^n)=0$ for $n>0$,
  and $C=R[t,s]$ with $\Delta(t^ns^m)= \Sigma_{i=0}^n {n \choose i}  t^is^m \ot t^{n-i}s^m$ for any basis element $t^ns^m$ in $R[t,s]$. The action of $H$ on $C$ is given by $t^jx^i=t^{i+j}$ for all $j\geq 0$ and $s^j x^i= 0$ if $j>0$. The action on all mixed terms are zero, i.e $( t^n s^m ) x^k=0$ if $m>0$. With respect to this action $C$ is a $H$-module coalgebra.
  We introduce the total cointegral $f\colon R[t,s]\longrightarrow R[x]$  given by
  \begin{equation}
  f(t^i s^j)= \begin{cases}
0  \quad  \text{if} \quad&\text{if $j> 0$}\\
x^i  \quad  \text{if} \quad   &\text{if $j=0$}.\\
\end{cases}
 \end{equation}

  It is easy to check that $f$ is a $H$-module coalgebra map. Then  any Hopf Galois coextension $C(D)_H$ is cleft.  The coalgebra $C=R[t,s]$ is a right-right Yetter-Drinfeld module over $R[x]$ by the action that we described and the following coaction,
  \begin{equation}
    t^ns^m\longmapsto \begin{cases}
0  \quad  \text{if} \quad&\text{if $m> 0$}\\
\sum_{i=j}^n\sum_{j=0}^i {n \choose i}{i \choose j}   t^{i-j}\ot x^{n-i-j}  \quad   \quad   &\text{if $m=0$}.\\
\end{cases}
  \end{equation}

 Furthermore $C^D$ is a right-left  anti yetter-Drinfeld module over $R[x]$ by the original action of $R[x]$ over $C=R[s,t]$ and the following coaction
 \begin{equation}
    t^ns^m\longmapsto \begin{cases}
0  \quad  \text{if} \quad&\text{if $m> 0$}\\
\sum_{i=j}^n\sum_{j=0}^i {n \choose i}{i \choose j} x^{i+j-n}\ot  t^{i-j}  \quad   \quad   &\text{if $m=0$}.\\
\end{cases}
  \end{equation}

}
\end{example}


\end{document}